\documentclass[12pt,a4paper]{article}
\usepackage{amsmath,amssymb,amsthm,hyperref}
\usepackage{graphicx}
\usepackage{version} 
\usepackage{color}
\usepackage[bf]{caption}     
\usepackage{geometry}        
\usepackage{bm}  
\usepackage[utf8]{inputenc}
  \newtheorem{theorem}{Theorem}[section]
 
 \newtheorem{claim}[theorem]{Claim}
 \newtheorem{lemma}[theorem]{Lemma}
 
 \newtheorem{example}[theorem]{Example}
 {\theoremstyle{definition}
 \newtheorem{definition}[theorem]{Definition}
 \newtheorem{remark}[theorem]{Remark}
 
 \newtheorem{question}[theorem]{Question}
 
 }
\renewcommand{\P}{\ensuremath{\mathbb P}} 
\newcommand{\beq}{\begin{equation}}
\newcommand{\eeq}{\end{equation}}

\usepackage{enumitem}

\title{Random constructions for translates of non-negative functions}
\author{Zolt\'an Buczolich\thanks{\scriptsize
Research supported by the Hungarian National Research, Development and Innovation Office--NKFIH, Grant  124003.
},
Department of Analysis, ELTE E\"otv\"os Lor\'and\\
University, P\'azm\'any P\'eter S\'et\'any 1/c, 1117 Budapest, Hungary\\
email: buczo@cs.elte.hu\\
{\tt www.cs.elte.hu/\hbox{$\sim$}buczo}\\
ORCID Id: 0000-0001-5481-8797
 \medskip\\
  Bruce Hanson, Department of Mathematics,\\ Statistics and Computer Science,\\ St.\ Olaf College,
Northfield, Minnesota 55057, USA\\
{email:} \texttt{hansonb@stolaf.edu}
\medskip\\
 Bal\'azs Maga\thanks{\scriptsize This author was supported by the \'UNKP-17-2 New National Excellence of the Hungarian Ministry of Human Capacities, and by the Hungarian National Research, Development and Innovation Office–NKFIH, Grant 124749.},
Department of Analysis, ELTE E\"otv\"os Lor\'and\\
University, P\'azm\'any P\'eter S\'et\'any 1/c, 1117 Budapest, Hungary\\
 email: magab@cs.elte.hu \\{\tt www.cs.elte.hu/\hbox{$\sim$}magab}
 \medskip\\
and
 \medskip\\
 G\'asp\'ar V\'ertesy\thanks{\scriptsize  This author was supported by the Hungarian National Research, Development and Innovation Office–NKFIH, Grant 124749.
},
 Department of Analysis, ELTE E\"otv\"os Lor\'and\\
University, P\'azm\'any P\'eter S\'et\'any 1/c, 1117 Budapest, Hungary\\
email: vertesy.gaspar@gmail.com\
}
\date{\today}
\begin{document}
\maketitle

\medskip

\medskip

\medskip


\begin{abstract}
 Suppose $\Lambda$ is a discrete infinite set of nonnegative real numbers.
We say that $ {\Lambda}$ is  type $2$ if the series $s(x)=\sum_{\lambda\in\Lambda}f(x+\lambda)$
 does not satisfy a zero-one law. This means that we can find a non-negative measurable ``witness function"
 $f: {\ensuremath {\mathbb R}}\to [0,+ {\infty})$ such that both the convergence set $C(f, {\Lambda})=\{x: s(x)<+ {\infty} \}$
and its complement the divergence set $D(f, {\Lambda})=\{x: s(x)=+ {\infty} \}$ are of positive Lebesgue measure.
 If $ {\Lambda}$ is not  type $2$ we say that $ {\Lambda}$ is  type $1$.

 The main result of our paper answers a question raised by Z. Buczolich,
 J-P. Kahane, and D. Mauldin. By a random construction we show that one can always choose
 a witness function which is the characteristic function of a measurable set.

 We also consider the effect on the type of a set $ {\Lambda}$ if we randomly delete 
 its elements.

 Motivated by results concerning  weighted sums $\sum c_n f(nx)$ and the Khinchin conjecture, we also discuss some results about weighted sums\\
 $\sum_{n=1}^{\infty}c_n f(x+\lambda_n)$.

  \end{abstract}


\section{Introduction}\label{*secintro}
The original research leading to this paper started with questions
concerning convergence properties of series of the type
$\sum_{n=1}^{{\infty}}f(nx)$ for nonnegative measurable functions $f$.

First there were results about the periodic case.
This is the case where $f: {\ensuremath {\mathbb R}}\to {\ensuremath {\mathbb R}}$ is a periodic measurable function and
without limiting generality we can assume that its period
$p=1$.

Results of Mazur and Orlicz in \cite{[MO]}
imply that if the periodic function
$f$ is the characteristic function of a set of
 positive (Lebesgue) measure,
then for almost every $x$ we have
$\sum_{n}f(nx)= {\infty}$.
Thus, in the periodic case we have a zero-one law:
the series either converges or diverges almost everywhere.

In this case it is more interesting to consider the
Ces{{$ \grave{\hbox{a}}$}}ro $1$ means of the partial sums
of our series.
A famous problem is the
{\it Khinchin conjecture \cite{[Kh]} (1923):}

{\it Assume that
$E {\subset} (0,1)$
is a measurable set and
$f(x)=\chi_{E}(\{x\})$,
where $\{x\}$ denotes the fractional part of $x$.
Is it true that for almost every $x$
$$\frac{1}{k}\sum_{n=1}^{k} f(nx)\to  {\mu}(E)?$$}
(In our paper $ {\mu}$ denotes the Lebesgue measure.)

Even at the time of the statement of the Khinchin conjecture it was a known result  of
H. Weyl \cite{[We]},
 that there is a positive answer to the above question if
$f$ is Riemann integrable.

However in 1969
Marstrand \cite{[M]}
showed that the Khinchin conjecture is not true.
Other counterexamples were given by J. Bourgain \cite{[Bou]} by using his entropy method and by A. Quas and M. Wierdl \cite{[QW]}.
For further results related to the Khinchin conjecture  we also refer to
\cite{[Beck]} and \cite{[Bee]}.

In the non-periodic measurable case there was a question of
Heinrich von Weizs\"aker \cite{[HW]} concerning a zero-one law:

{\it Suppose $f:(0,+ {\infty})\to {\ensuremath {\mathbb R}}$ is a measurable function.
Is it true that
$\sum_{n=1}^{{\infty}}f(nx)$
either converges (Lebesgue) almost everywhere or diverges almost everywhere, i.e.
is there  a zero-one law for $\sum f(nx)$?}

J. A. Haight in \cite{[H1]} also considered a similar question
and his results implied that there exists a measurable set
$H {\subset} (0, {\infty})$ such that
 if
$f(x)=\chi_{H}(x)$,
the characteristic function of
 $H$ then
$\int_{0}^{{\infty}}f(x)dx= {\infty}$
and $\sum_{n=1}^{{\infty}}f(nx)< {\infty}$
everywhere.

In  \cite{[BM1]}
Z. Buczolich and D. Mauldin answered the Haight--Weizs\"aker problem.

\begin{theorem}\label{*HWth} There exists a
measurable function $f:(0,+ {\infty})\to \{0,1\}$ and
two nonempty intervals $I_{F}, \ I_{{\infty}} {\subset} [{1\over 2},1)$
such that for every $x\in I_{{\infty}}$ we have $\sum_{n=1}^{{\infty}}f(nx)=+ {\infty}$
and for almost every $x\in I_{F}$ we have $\sum_{n=1}^{{\infty}}f(nx)<+ {\infty}.$
The function $f$ is the characteristic function of an open set
$E$.\end{theorem}

Later with Jean-Pierre Kahane in papers \cite{[BKM1]} and \cite{[BKM2]}
we considered a more general, additive version of the Haight--Weizs\"aker problem.
After a simple exponential/logarithmic substitution and change of variables
one obtains
 almost everywhere
convergence questions for the series
$\sum_{{\lambda}\in {\Lambda}}f(x+ {\lambda})$ for non-negative functions
defined on $ {\ensuremath {\mathbb R}}$.
Taking  $\Lambda=\{\log n:
n=1,2,... \}$
we obtain an ``additive" version of the question answered in
Theorem \ref{*HWth}.
Of course, one can consider other infinite, unbounded sets $ {\Lambda}$,
different from $\{\log n:
n=1,2,... \}$.

In fact, in \cite{[H2]} Haight already considered this more general case  in the original ``multiplicative" setting.
He proved, that for any countable set $ {\Lambda} {\subset}[0,+ {\infty})$   such that
the only accumulation point of $\Lambda$ is $+ {\infty}$ there exists a measurable set $E {\subset} (0,+ {\infty})$
such that choosing $f=\chi_{E}$
we have
$\sum_{{\lambda}\in {\Lambda}}f( {\lambda} x)< {\infty}$, for any $x$ but
$\int_{{\ensuremath {\mathbb R}}^+} f(x)dx= {\infty}.$

In \cite{[BKM1]}  we introduced the notion of  type $1$ and type $2$ sets.
Given $\Lambda$ an unbounded, infinite discrete set of nonnegative numbers, and
a measurable $f: {\ensuremath {\mathbb R}}\to [0,+ {\infty})$, we consider the sum
$$s(x)=s_{f}(x)=\sum_{\lambda\in\Lambda}f(x+\lambda),$$
and the complementary subsets of $ {\ensuremath {\mathbb R}}$:
$$C=C(f,\Lambda)=\{x: s(x)< {\infty}\},\qquad
D=D(f,\Lambda)=\{x:s(x)= {\infty}\}.$$

\begin{definition}\label{def1} The set $\Lambda$ is type $1$ if, for every $f$, either
$C(f,\Lambda)= {\ensuremath {\mathbb R}}$ a.e.\!\! or $C(f,\Lambda)= {\emptyset}$ a.e. (or equivalently
$D(f,\Lambda)= {\emptyset}$ a.e. or $D(f,\Lambda)= {\ensuremath {\mathbb R}}$ a.e.). Otherwise, $\Lambda$
is type $2$. \end{definition}

That is, for type $1$ sets we have a ``zero-one" law for the almost everywhere
convergence properties of the series $\sum_{\lambda\in\Lambda}f(x+\lambda)$,
while for type $2$ sets the situation is more complicated.

\begin{example}\label{*exdyada}
Set $\Lambda=\cup_{k \in {{{\ensuremath {\mathbb N}}}}}\Lambda_{k},$ where $\Lambda_{k}=2^{-
k} {{{\ensuremath {\mathbb N}}}}\cap [k,{k+1}).$
In Theorem 1 of \cite{[BKM1]} it is proved that $\Lambda$ is type $1$.
In fact, in a slightly more general version it is shown that
if $(n_{k})$ is an increasing sequence of
positive integers  and $\Lambda=\cup_{k\in  {{{\ensuremath {\mathbb N}}}}}\Lambda_{k}$ where $\Lambda_{k}=2^{-
k} {{{\ensuremath {\mathbb N}}}}\cap [n_{k},n_{k+1})$
then $\Lambda$ is type $1$.
\end{example}

\begin{example}\label{*exdyadb}
Let
$(n_{k})$ be a given increasing sequence of positive integers. 
 By Theorem 3 of \cite{[BKM1]} 
there is
an increasing sequence of integers $(m(k))$ such that the set
$\Lambda=\cup_{k\in  {{{\ensuremath {\mathbb N}}}}}\Lambda_{k}$ with $\Lambda_{k}=2^{-m(
k)} {{{\ensuremath {\mathbb N}}}}\cap [n_{k},n_{k+1})$ is type $2$.
\end{example}

According to Theorem 6 of \cite{[BKM1]}
type $2$ sets  form a dense
open subset in the box topology of discrete sets while type $1$ sets  form a closed
nowhere dense set. Therefore type $2$ is typical in the Baire
category sense in this topology. Indeed, this is in line with our experience
that it is more difficult to find and verify that a $ {\Lambda}$ is type $1$.

\begin{definition}\label{asydens} The unbounded, infinite  discrete set $\Lambda=\{{\lambda}_{1}, {\lambda}_{2},... \}$, $ {\lambda}_{1}< {\lambda}_{2}<...$ is asymptotically dense if $d_{n}= {\lambda}_{n}- {\lambda}_{n-1}\to 0$, or equivalently:
$$\forall a>0,\quad \lim_{x\to\infty}\#(\Lambda\cap [x,x+a])=\infty.$$

If  $\Lambda$ is not asymptotically dense we say that it is asymptotically lacunary.
\end{definition}

We recall Theorem 4 from of \cite{[BKM1]}
\begin{theorem}\label{*thlacu}
If $\Lambda$ asymptotically lacunary, then $\Lambda$ is
type $2$. Moreover, for some $f\in C_{0}^{+}( {{{\ensuremath {\mathbb R}}}})$, there exist
intervals $I$ and $J$, $I$ to the left of $J$, such that $C(f,\Lambda)$
contains $I$ and $D(f,\Lambda)$ contains $J.$
\end{theorem}

We denote the non-negative continuous functions on $ {\ensuremath {\mathbb R}}$ by $C^{+}( {\ensuremath {\mathbb R}})$,
and if, in addition these functions tend to zero as $x\to + {\infty}$ they belong to
$C^{+}_{0}( {\ensuremath {\mathbb R}})$.

In \cite{[BKM1]} we gave some necessary and some sufficient conditions
for  a set $ {\Lambda}$ to be type $2$. A complete characterization of type $2$
sets is still unknown. We recall here
from \cite{[BKM1]}
the theorem concerning the Haight--Weizs\"aker problem. This contains the additive version of the result of Theorem \ref{*HWth}
with some additional information.

\begin{theorem}\label{*BKMHWth} {The set $\Lambda=\{\log n:
n=1,2,... \}$ is type $2$. Moreover, for some $f\in
C_{0}^{+}( {\ensuremath {\mathbb R}}),$ $C(f, {\Lambda})$ has full measure on the half-line
$(0,\infty)$ and $D(f, {\Lambda})$ contains the half-line $(-\infty,0)$.
If for each $c, \int_c^{+\infty}e^yg(y)dy < +\infty$, then $C(g,\Lambda) =  {\ensuremath {\mathbb R}}$ a.e.
If $g\in C_{0}^{+}( {\ensuremath {\mathbb R}})$ and $C(g,\Lambda)$ is not of the first (Baire) category, then
$C(g, {\Lambda})= {\ensuremath {\mathbb R}}$ a.e. Finally, there is some $g\in C_{0}^+( {\ensuremath {\mathbb R}})$ such that $C(g,\Lambda) =  {\ensuremath {\mathbb R}}$ a.e. and
$\int_0^{+\infty}e^yg(y)dy = +\infty$.
}
\end{theorem}

From the point of view of our current paper the following question
(QUESTION 1 in \cite{[BKM1]}) is the most relevant:
\begin{question}\label{*q1bkm1}
Is it true that $\Lambda$ is type $2$ if and only if there
is a $\{0,1\}$ valued measurable function $f$ such that both
$C(f,\Lambda)$ and $D(f,\Lambda)$ have positive Lebesgue measure?
\end{question}

 In Section \ref{*secwit} we give a positive answer to this question.
 This result is very useful if one tries to study type $2$
sets. In later sections of this paper and in another forthcoming paper
\cite{[BHMVa]} one can see applications of this result.

In Section \ref{*secdel} we take some type $1$ sets from Example
\ref{*exdyada}  and   investigate the effect of
random deletion of elements with probability $q$.
We see in Theorem \ref{prob p theorem} that in the basic case
of Example \ref{*exdyada}, that is, when $n_{k}=k$ after randomization
$ {\Lambda}$ stays type $1$.
However in Theorem \ref{*randt2} we show that for some other $n_{k}$s one
can turn a type $1$ set into a type $2$ set by random deletions.


In \cite{[BKM1]} two questions were  stated.
We have already mentioned Question 1, which is the main motivation
for our paper.
Question 2 was the following:
{\it Given open sets $G_1$ and $G_2$ when is it possible
to find $\Lambda$ and $f$ such that $C(f,\Lambda)$ contains $G_1$ and
$D(f,\Lambda)$ contains $G_2$?}
This question was essentially answered in our recent paper
 \cite{[BMV]}.

In the periodic case, corresponding to the Khinchin conjecture,
several papers
considered weighted averages $\sum c_{k}f(n_{k} x)$.
See for example \cite{[ABW]}, \cite{[BWa]},  and  \cite{[BWb]}.
This motivates the following definition:

\begin{definition}\label{*defctype}
We say that an asymptotically dense set $\Lambda$ is $\pmb{c}$-type $2$ with respect to the positive sequence $\pmb{c}=(c_n)_{n=1}^{\infty}$,
 if there exists a nonnegative measurable ``witness" function $f$ such that the series $s_{\pmb{c}}(x)=s_{\pmb{c},f}(x)=\sum_{n=1}^{\infty}c_n f(x+\lambda_n)$ does not converge almost everywhere and does not diverge almost everywhere either.
Of course, those $ {\Lambda}$ which are not $\pmb{c}$-type $2$ will be called
$\pmb{c}$-type $1$.
\end{definition}

In the sense of our earlier definition, $\Lambda$ is type $2$ if it is $\pmb{c}$-type $2$ with respect to $c_n\equiv 1$.
We also say in this case that $ {\Lambda}$ is $\pmb{1}$-type $2$. For the corresponding convergence and divergence sets
 we introduce the notation $C_{\pmb{c}}(f,\Lambda)$ and $D_{\pmb{c}}(f,\Lambda)$. 
 
In Theorem
 \ref{*thpmbc1}  of Section \ref{*secpmbc}
 we see that
 if a set $\Lambda$ is $\pmb{1}$-type $2$, then it is $\pmb{c}$-type $2$ with respect to any positive sequence $\pmb{c}$.
The key property behind this theorem is the fact that for  $\pmb{1}$-type $2$
sets there is a always a witness function which is a characteristic function
according to the result of Theorem \ref{*thwit}. This motivates the following definition.

 \begin{definition}\label{*defchi}
A positive sequence $\pmb{c}$ is a $\chi$-sequence
if for any $\pmb{c}$-type  $2$  set $ {\Lambda}$
there is always a characteristic function to witness this property.
\end{definition}

It would be interesting to see whether Theorem  \ref{*thpmbc1} holds for
all $\chi$-sequences.

It is also worthful to notice that there exist sequences which are not $\chi$-sequences. Indeed, if $\sum c_n$ converges, then for any function $f$ bounded by $K$ we have $s_{\pmb{c},f}(x)=\sum_{n=1}^{\infty}c_{n}f(x+\lambda_{n})\leq \sum_{n=1}^{\infty}c_{n}K$,
and hence $s_{\pmb{c},f}$ converges everywhere. 
On the other hand, by Theorem \ref{*thpmbc1} there are  
$\pmb{c}$-type $2$ sets $\Lambda$, in this case with unbounded
witness functions.

 Hence it is also a natural question for further research to characterize $\chi$-sequences.

Finally, in Theorem \ref{*univt2seq} we prove that there are sequences
$\pmb{c}$ such that
   every discrete set $\Lambda$ is $\pmb{c}$-type  $2$.

\section{Preliminaries}\label{*secprel}

In the proof of Proposition 1 of \cite{[BKM1]} we used a simple argument
based on the Borel--Cantelli lemma which we state here as the following lemma.

\begin{lemma}\label{*lemmodBKM}
Suppose that $ {\Lambda}$ is  type $2$ and $f$ is a bounded witness function for $ {\Lambda}$.
 If we modify $f$ on a set $E$ such that $\mu(E\cap(x,\infty))\leq{\epsilon(x)}$ where $\epsilon(x)$ is a positive decreasing function tending to $0$ at infinity, and satisfying
\begin{equation}\label{*modeq}
\sum_{l\in\mathbb{N}}\epsilon(l-K)\#(\Lambda\cap[l,l+1))<\infty,
\end{equation}
then the convergence and divergence sets in $[-K,K]$  for the modified function $\widetilde{f}$ do not change apart from a set of measure $0.$
\end{lemma}

\section{Characteristic functions are witness functions for type $2$}
\label{*secwit}

\begin{theorem}\label{*thwit} Suppose that $\Lambda$ is type $2$, that is there exists a measurable witness function $f$ such that both $D(f,\Lambda)$ and $C(f,\Lambda)$ have positive measure. Then there exists a witness function $g$ which is the characteristic function of an open set and both $D(g,\Lambda)$ and $C(g,\Lambda)$ have positive measure.\end{theorem}

\begin{proof}

First we observe that it is sufficient to find a suitable $g$ which is the characteristic function of a measurable set: then we can modify it on a set of finite measure which does not change the measure of $D(g,\Lambda)$ and $C(g,\Lambda)$.

Fix bounded sets $D\subset D(f,\Lambda)$ and $C\subset C(f,\Lambda)$ of positive measure, satisfying $C\cup D\subset[-K,K]$ for some $K\in\mathbb{N}$.
We will suitably modify the function $f$ by a sequence of steps such that the function obtained after each step satisfies the condition concerning the measures of $D$ and $C$.
  Consider the intervals $I_1=(-\infty,1)$, $I_2=[1,2)$, $I_3=[2,3)$, .... For $n=1,2,...$
 we will choose sufficiently small real numbers $\varepsilon_n>0$ and define $f_0$ such that $f_0(x)=f(x)+\varepsilon_k$ in $I_k$,
  $C\subset C(f_0,\Lambda)$ and $D\subset D(f_0,\Lambda)$. As $f_0>f$, the second condition is obviously satisfied.
 Furthermore, as $\Lambda$ is discrete and bounded from below, for fixed $n$ there is a bounded number of $\lambda_i$s with $\lambda_i\in I_n-[-K,K]$. Thus by choosing $\varepsilon_n$ small enough, we can ensure that
$$
\sum_{x+\lambda_i\in I_n}f_0(x+\lambda_i)<\frac{1}{2^n}+\sum_{x+\lambda_i\in I_n}f(x+\lambda_i)\text{  for any $x\in C\subset[-K,K]$.  }
$$
 As a consequence, for any $x\in{C}$ we have $s_{f_0}(x)<s_f(x)+1<\infty$, thus $C\subset C(f_0,\Lambda)$, as we stated. Hence $f_0$ is a function such that both $D(f_0,\Lambda)$ and $C(f_0,\Lambda)$ have positive measure,
 and $f_0$ is
 bounded away from zero on any interval of the form $(-\infty,t)$.

Now take $f_1(x)=\min(f_0(x),1)$. For any $x\in D(f_0,\Lambda)$, if the sum $\sum_{\lambda\in\Lambda}f_0(x+\lambda)$ contains infinitely many terms which are at least 1, then these terms immediately guarantee that
\begin{equation} \label{*f/div}
\sum_{\lambda\in\Lambda}f_1(x+\lambda)=\infty.
\end{equation}

On the other hand, if there are only finitely many such terms, then the sums associated to $f_1$ and $f_0$ differ only in these finitely many terms, which also yields \eqref{*f/div}. Then we have $D(f_1,\Lambda)=D(f_0,\Lambda)$, and consequently $C(f_1,\Lambda)=C(f_0,\Lambda)$. Thus we obtained a function $f_1$ which is bounded by 1. Moreover, both $D(f_1,\Lambda)$ and $C(f_1,\Lambda)$ have positive measure, and $f_1$ is bounded away from zero on any interval of the form $(-\infty,t)$.

Given $f_1$, we can construct a function $f_2$ with a rather simple range. Namely, for any $x$ we choose $k_x\in\mathbb{N}$ such that $\frac{1}{2^{k_x}}< f_1(x)\leq\frac{1}{2^{k_x-1}}$. As the range of $f$ is contained by $(0,1]$, there is such a $k_x$. Now take $f_2(x)=\frac{1}{2^{k_x}}$. Since we have $\frac{1}{2}f_1\leq f_2<f_1$, we may deduce $C(f_2,\Lambda)=C(f_1,\Lambda)$ and $D(f_2,\Lambda)=D(f_1,\Lambda)$. Moreover, as $f_1$ is bounded away from zero on any interval of the form $(-\infty,t)$, we have that $f_2$ has finite range in each such interval and vanishes nowhere. We can also assume $f_2\equiv{1}$ in $(-\infty,0)$ .
As $\Lambda$ is discrete and bounded from below the convergence and divergence sets remain the same.

 Consider now the interval $[k-1,k)$ for $k\in\mathbb{N}$. The range of $f_2$ is finite in $[k-1,k)$, let it be $\{c_1,c_2,...,c_l\}$. Now for any level set $\{f_{2}=c_i\}$ we can define a relatively open set $U_i\subset [k-1,k)$ such that $\{f_{2}=c_i\}\cap[k-1,k)\subset U_i$ and
$$
\mu(U_i)<\mu(\{f_{2}=c_i\}\cap[k-1,k))+\frac{\delta_k}{2^i}
$$
for some $\delta_k$ to be chosen later.
 Each $U_i$ is a countable union of intervals. By choosing a sufficiently large finite subset of these intervals we can obtain a set $V_i$ such that
$$
\mu(U_i)<\mu(V_i)+\frac{\delta_k}{2^i}.
$$
The intervals forming $V_i$ are relatively open in $[k-1,k)$. Hence by adding finitely many points to each of them we can get sets
 $V_i'$ which are finite unions of intervals of the form $[x,y)$. Finally, let $V_1^*=V_1'$, and for $i=2,...,l$ let
$$
V_i^*=V_i'\setminus\left(\bigcup_{j=1}^{i-1}V_j^*\right).
$$
Then the sets $V_i^*$ are disjoint and 
each of them
is  a finite union of intervals of the form $[x,y)$ as such intervals form a semialgebra. Moreover, the complement $V^*$ of their union in $[k-1,k)$ is also a set of this form. We define $f_3$ using these sets: on $V_i^*$ let $f_3=c_i$, and on $V^*$ let $f_3=c_1$.

When we redefine our function in $V_i^*$ we modify it in a set of measure at most $\frac{\delta_k}{2^i}$, and when we redefine it in $V^*$ we modify it in a set of measure at most $\sum_{i=1}^{l}\frac{\delta_k}{2^i}$.
Hence $f_2$ and $f_3$ can differ only in a set of measure at most
$$
2\delta_k\sum_{i=1}^{l}\frac{1}{2^i}< 2\delta_k.
$$

Put $\epsilon(x)=\sum_{k\geq{x}}\delta_k$.
If we choose a sufficiently rapidly decreasing sequence $(\delta_k)$ then
 we can ensure that
\begin{equation}\label{*modeqq}
\sum_{l\in\mathbb{N}}\epsilon(l-K)\#(\Lambda\cap[l,l+1))<\infty.
\end{equation}
Since \eqref{*modeqq} is assumption \eqref{*modeq} of Lemma \ref{*lemmodBKM},
if we define $f_3$ in each of the intervals $[k-1,k)$ using the previous procedure, then the convergence and divergence sets are the same for $f_{2}$ and $f_{3}$
 almost everywhere in $[-K,K]$.
 Moreover, the range of $f_3$ is a subset of the range of $f_2$ in any interval $(-\infty,t)$, and each bounded interval can be subdivided into finitely many subintervals of the form $[a,b)$ such that $f_3$ is constant on each of these subintervals. Denote the family of all these subintervals in $\mathbb{R}$ by $\mathcal{I}$. We know that the sum $s_{f_3}$ diverges in $D$ apart from a null-set and converges in $C$ apart from a null-set. For the sake of simplicity we assume that the sum $s_{f_3}$ diverges in the entire set $D$ and converges in the entire set $C$: if that does not hold, we can modify our initial sets.
For ease of notation in the sequel we will denote  $f_3$  by $f$,
in fact it can be assumed that $f$ was originally of this form.

In the following step we replace the family of intervals $\mathcal{I}$ by a ``finer" family $\mathcal{J}$.
 Precisely, if $I\in\mathcal{I}$, first we subdivide it into sufficiently short subintervals $I_1,...,I_{m}$ of equal length such that for any $x\in C\cup D$ we have that $x+\lambda \in I_i$ for at most one $\lambda\in\Lambda$ for any $i=1,2,...,{m}$. 
In order to avoid technical complications, we define them to be closed from the left and open from the right, hence guaranteeing that they are disjoint. As $C\cup D$ is bounded, it is clear that this is possible. The family $\mathcal{J}$ will consist of all the previous 
short intervals for each $I\in \mathcal{I}$.

Now we define a sequence of random variables. Consider the intervals in $\mathcal{J}$ in increasing order: $J_1,J_2,...$. Let $J_n\in\mathcal{J}$. Then we have that $f=2^{-\kappa_n}$ on
$J_n$ for some $\kappa_n\in\mathbb{N}$. We define the sequence $(X_n)$ of random variables such that they are independent and $X_n=1$ with probability $2^{-\kappa_n}$, otherwise $X_n$ is 0. By Kolmogorov's consistency theorem such random variables can be defined on a suitable probability measure space $\Omega$. Given these random variables, we can define a random characteristic function: for any $\omega\in\Omega$ and $x\in\mathbb{R}$ let $g(\omega,x)=X_n(\omega)$ if $x\in{J_n}$.

We claim that almost surely, that is for $\mathbb{P} $ almost every $ {\omega}$
$$
s_g(\omega,x)=\sum_{\lambda\in\Lambda}g(\omega,x+\lambda)
$$
converges in $C$ apart from a $ {\mu}$ null-set, and diverges in $D$ apart from a
$ {\mu}$ null-set. Proving the claim finishes the proof of the theorem as we can define $g=g(\omega)$ for one of the $\omega$s  of these almost sure events.

First let us consider the behaviour of $s_g(\omega,x)$ in $D$. Fix $x\in{D}$. Also fix $\lambda\in\Lambda$. Let us observe that
$$
\mathbb{P}(g(\omega,x+\lambda)=1)=f(x+\lambda).
$$
Indeed if $f(x+\lambda)=2^{-\kappa}$ for some $\kappa\in\mathbb{N}$, we have that $x+\lambda$ lies in an interval $J_n$ where $f=2^{-\kappa}$, thus
$$
\mathbb{P}(g(\omega,x+\lambda)=1)=\mathbb{P}(X_n(\omega)=1)=2^{-\kappa}=f(x+\lambda),
$$
as we claimed. As a consequence, by the definition of $D$ for $x\in{D}$ we clearly have that
\begin{equation}\label{*pgo}
\sum_{\lambda\in\Lambda}\mathbb{P}(g(\omega,x+\lambda)=1)=\sum_{\lambda\in\Lambda}f(x+\lambda)=\infty.
\end{equation}
Observe the events appearing in the leftmost expression. For fixed $\lambda$, the value $g(\omega,x+\lambda)$ depends on at most one of the independent random variables $X_1,X_2,...$. Moreover, by the procedure by which we replaced $\mathcal{I}$ by $\mathcal{J}$,
 for fixed $n$ and $x$ the random variable $X_n$ affects at most one of the values $g(\omega,x+\lambda)$, $\lambda\in  {\Lambda}$.
 Thus by the independence of $(X_n)$, for fixed $x\in{D}$ the events
 \begin{equation}\label{*alx}
 A_{\lambda,x}=\{\omega: g(\omega,x+\lambda)=1\}
 \end{equation}
  are also independent. As the series of their probabilites diverges, by the second Borel--Cantelli lemma we have that with probability one infinitely many of them occur, which is equivalent to the fact that $s_g(\omega,x)=\infty$. Thus for any fixed $x\in{D}$ we obtain $s_g(\omega,x)=\infty$ almost surely.

Now let us define $\Omega_D=\{(\omega,x)\text{  :  } x\in{D}\text{  ,  } s_g(\omega,x)=\infty\}\subset{\Omega\times D}$. We claim that it is measurable. Indeed, let
$$
\Omega_{\lambda_k}=\left\{(\omega,x):g\left(\omega,x+\lambda_k\right)=1\right\}.
$$
It would be sufficient to verify that such a set is measurable as
$$
\Omega_D=\limsup_{k\to\infty}\Omega_{\lambda_k}=\bigcap_{n=1}^{\infty}\bigcup_{k=n}^{\infty}\Omega_{\lambda_k}
$$
clearly holds. Now we simply observe that for fixed $k$ the set $D$ can be subdivided using finitely many intervals in 
each of which $g(\omega,x+\lambda_k)$ depends only on one of the random variables in the sequence $(X_n)$. Consequently, $\Omega_{\lambda_k}$ can be written as a finite union of rectangles, hence it is measurable in the product space, which verifies our claim: $\Omega_D$ is measurable.
 By the earlier observations we obtain for its measure
$$
\mu_{\Omega\times{D}}(\Omega_D)=\int_{\Omega\times{D}} \mathbf{1}_{\Omega_D}d\omega dx=\int_D\left( \int_{\Omega}\mathbf{1}_{\Omega_D}d\omega\right) dx=\int_{D} 1 dx=\mu(D).
$$
Hence $\Omega_D$ is of full measure in the product space $\Omega\times{D}$. Thus
almost surely $s_g(\omega,x)$ diverges in $D$ apart from a null-set, that is $$\mathbb{P}\left(\omega:s_g(\omega,x)=\infty \text{  for a.e.  } x\in{D}\right)=\mathbb{P}(\Omega_D')=1.$$

The behaviour of $s_g(\omega,x)$ in $C$ can be treated similarly. More precisely, the beginning of the argument up to \eqref{*pgo} can be repeated and in place of \eqref{*pgo} we obtain
\begin{equation}
\sum_{\lambda\in\Lambda}\mathbb{P}(g(\omega,x+\lambda)=1)=\sum_{\lambda\in\Lambda}f(x+\lambda)<\infty
\end{equation}
for fixed $x\in{C}$. We do not even have to check independence in this case;
 we can simply apply the first Borel--Cantelli lemma which tells us that with probability $1$ only finitely many of the events $A_{\lambda,x}$
in \eqref{*alx}
 occur for $x\in{C}$, hence  for fixed $x\in{C}$ almost surely $s_g(\omega,x)<\infty$. The conclusion is also similar: the measure of $\Omega_C=\{(\omega,x)\text{  :  } x\in{C}\text{  ,  }s_g(\omega,x)<\infty\}\subset{\Omega\times C}$ equals
$$
\mu_{\Omega\times{C}}(\Omega_C)=\int_{\Omega\times{C}} \mathbf{1}_{\Omega_C}d\omega dx=\int_C\left( \int_{\Omega}\mathbf{1}_{\Omega_C}d\omega\right) dx=\int_{C} 1 dx=\mu(C).
$$
(The measurability of $\Omega_C$ can be verified analogously to that of $\Omega_D$.)
Hence $\Omega_C$ is of full measure in the product space $\Omega\times{C}$. Thus $s_g(\omega,x)$ converges in $C$ apart from a null-set almost surely, as we stated, that is $$\mathbb{P}\left(\omega:s_g(\omega,x)<\infty \text{  for a.e.  } x\in{C}\right)=\mathbb{P}(\Omega_C')=1.$$ This concludes the proof: the choice $g=g(\omega)$ for any $\omega\in\Omega_D'\cap\Omega_C'$ satisfies the claims of the theorem, thus there exists a satisfactory characteristic function.
\end{proof}

\section{Randomly deleted points from $ {\Lambda}$}\label{*secdel}

\begin{lemma}\label{*lemmart} Assume that $C\subset[0,1)$ is Lebesgue measurable, that is $C\in\mathcal{L}[0,1)$. Then for almost every $x\in[0,1)$ we have
\begin{equation}\label{*lemmarteq}
\lim_{n\to -\infty}\frac{\#((x+2^{n}\mathbb{Z})\cap{C})}{2^{-n}}=\mu(C).
\end{equation}
\end{lemma}

\begin{proof}
 Consider the measurable function $\mathbf{1}_C$ and the negatively indexed increasing sequence of $\sigma$-algebras $\mathcal{F}_n=\{(A+2^{n}\mathbb{Z})\cap [0,1): A\in\mathcal{L}([0,1))\}$, $n\in- {\ensuremath {\mathbb N}}$.
Moreover, denote by $\mathcal{F}_{-\infty}$ their intersection. By Lebesgue's density theorem one can easily see that $\mathcal{F}_{-\infty}$ contains only full measure sets and null-sets.
Hence the conditional expectation
$\mathbb{E}(\mathbf{1}_C|\mathcal{F}_{-\infty})$ is almost everywhere constant, therefore it equals $\mathbb{E}(\mathbf{1}_C)=\mu(C)$. On the other hand, by Theorem 5.6.3 in \cite{D} about backwards martingales we know that \begin{equation}\label{*condlim}
\mathbb{E}(\mathbf{1}_C|\mathcal{F}_{n})\to\mathbb{E}(\mathbf{1}_C|\mathcal{F}_{-\infty})
\end{equation}
almost surely as $n\to -\infty$.
Next we show that
\begin{equation}\label{*eoc}
\mathbb{E}(\mathbf{1}_C|\mathcal{F}_{n})(x)=\frac{\#((x+2^{n}\mathbb{Z})\cap{C})}{2^{-n}},\text{   for  }   {\mu}\text{  a.e. $x\in [0,1)$ for all  }n\in- {\ensuremath {\mathbb N}}.
\end{equation}
The function on the right-handside of \eqref{*eoc}
is defined for any $x\in  {\ensuremath {\mathbb R}}$. It is  Lebesgue measurable and 
invariant under translations by
values in $2^{n}\mathbb{Z}$, hence its restriction 
onto $[0,1)$
is clearly $\mathcal{F}_n$ measurable.
Suppose that $A'\in \mathcal{F}_n$.
We denote by $A$ the one periodic set obtained from
$A'$, that is $A=A'+ {\ensuremath {\mathbb Z}}$.
Then  ${A+2^n k}=
{A}$ for any $k\in  {\ensuremath {\mathbb Z}}$.
Moreover,
for any $n\in- {\ensuremath {\mathbb N}}$
 \begin{equation}\label{*sumi}
\int_{A'}\frac{\#((x+2^n\mathbb{Z})\cap C)}{2^{-n}}d {\mu}(x)=
2^n\int_{0}^{1}\Big (\sum_{k\in\mathbb{Z}}\mathbf{1}_{C+2^n k}(x)\Big )
\mathbf{1}_{A}(x)d {\mu}(x)
\end{equation}
$$
=2^n\int_{0}^{1}\Big (\sum_{k\in\mathbb{Z}}\mathbf{1}_{C+2^n k}(x)
\mathbf{1}_{A+2^n k}(x)\Big )
d {\mu}(x)
$$
$$=
2^n\sum_{m=0}^{2^{-n}-1}\left(\sum_{k\in\mathbb{Z}}\int_{m2^n}^{(m+1)2^n}\mathbf{1}_{C+2^n k}(x)\mathbf{1}_{A+2^n k}(x)d {\mu}(x)\right).
$$
However,
$$\sum_{k\in\mathbb{Z}}\int_{m2^n}^{(m+1)2^n}\mathbf{1}_{C+2^n k}(x)
\mathbf{1}_{A+2^n k}(x)d {\mu}(x)
=\sum_{k\in\mathbb{Z}}\int_{0}^{2^{n}}\mathbf{1}_{C+2^n k}(x)\mathbf{1}_{A+2^n k}(x)d {\mu}(x)$$ $$=
\sum_{k\in\mathbb{Z}} {\mu}(A\cap C\cap [-k2^{n},-(k+1)2^{n}))= {\mu}(A\cap C)=
 {\mu}(A'\cap C).
$$
Hence the left-hand side of \eqref{*sumi} equals $\mu(A'\cap C)$.
Since we have this property for any $A'\in \mathcal{F}_{n}$ we proved
\eqref{*eoc}.  Using this result in
\eqref{*condlim} and taking limit in \eqref{*condlim} we obtain \eqref{*lemmarteq}.
\end{proof}

Let
\begin{equation}\label{*tL}
\widetilde{\Lambda}=\bigcup_{k=1}^\infty (2^{-k} {\mathbb {N}} \cap[k,k+1)).
\end{equation}
 We know from
Example \ref{*exdyada} that $ {{\widetilde {\Lambda}}}$ is type $1$.

\begin{definition}\label{prob p}
Let $0<p<1$.  Then we say that $\Lambda\subset \widetilde{\Lambda}$ is chosen with probability $p$ from $\widetilde{\Lambda}$ if for each $\lambda \in \widetilde{\Lambda}$ the probability that $\lambda \in \Lambda$ is $p$.
That is, we consider $ {\Omega}=\{0, 1 \}^{{\ensuremath {\mathbb N}}}$
with the product measure $  \displaystyle  \P$ which is obtained as the product of the measures which assign probability $p$ to $\{1\}$ and $q=1-p$
to $\{0\}$. We order the elements of $\widetilde{\Lambda}$ in increasing order, that is
$\widetilde{\Lambda}=\{\lambda_{1}<\lambda_{2}<... \}$ and for an element $ {\omega}$ of our
probability space $ {\Omega}$ we assign the random set $ {\Lambda}_{{\omega}}$
which is obtained from $\widetilde{\Lambda}$ by keeping $\lambda_{k}$
if $ {\omega}_{k}$,
 the $k$th entry of $ {\omega}$ is $1$ and deleting it otherwise. To make this a little more precise we consider independent identically distributed random variables $X_{k}( {\omega})$ with values in $\{0, 1 \}$
with $X_{k}( {\omega})= {\omega}_{k}$. Then $\P(X_{k}=1)=p$, $\P(X_{k}=0)=q=1-p$
and we keep $\lambda_{k}$ in $ {\Lambda}_{{\omega}}$ if $X_{k}( {\omega})=1$.

We say that a property holds almost surely if the $\P$ measure  of those $ {\omega}$s for  which $ {\Lambda}_{{\omega}}$ has this property equals $1$.

For ease of notation often we omit the subscript $ {\omega}$ and we just speak about almost sure subsets $ {\Lambda}  {\subset} \widetilde{\Lambda}.$
\end{definition}

It is clear that almost surely if $\Lambda\subset \widetilde{\Lambda}$ is chosen with probability $p$ from $\widetilde{\Lambda}$ then $ {\Lambda}$ is an infinite discrete set.

\begin{theorem}\label{prob p theorem}
Suppose that $0<p<1$ and $\Lambda$ is chosen with probability $p$ from $\widetilde{\Lambda}=\cup_{k=1}^\infty (2^{-k} {\mathbb {N}} \cap[k,k+1))$.  Then almost surely $\Lambda$ is type 1.

\end{theorem}

\begin{lemma}\label{prob p lemma}
Suppose that $0<p<1$ and $\Lambda$ is chosen with probability $p$ from $\widetilde{\Lambda}$.  Then almost surely $\Lambda$ satisfies the following:

For every $L \in  {\mathbb {N}}$ there exists $N \in  {\mathbb {N}}$ such that for all $x \ge N$ we have
 \begin{equation}\label{p dense}
\#\Big (\Lambda\cap \Big [x,x+\frac{1}{2^L}\Big )\Big  ) > {p}\cdot{2^{J-L-2}},
 \end{equation}
where $J= \lfloor x  \rfloor$.
\end{lemma}

\begin{proof}
We will use the notation from Definition \ref{prob p}.
We consider $ {\Lambda}= {\Lambda}_{{\omega}}$ obtained from $\widetilde{\Lambda}$
by using the i.i.d. random variables $X_{k}=X_{k}( {\omega})= {\omega}_{k}$.

We recall from the standard Chebyshev’s
inequality proof of the Weak Law of Large Numbers (see for example
\cite[Ch. 2.2]{D}) that there exists a constant $C_{p}$
depending only on $p$ such that for any $K,n\in  {\ensuremath {\mathbb N}}$
\begin{equation}\label{*wlaw}
\P\Big \{\Big |\frac{X_{K+1}+...+X_{K+n}}{n}-p\Big |>p(1-2^{-1.9}) \Big \}<\frac{C_{p}}{n}.
\end{equation}

Observe that  $ {{\widetilde {\Lambda}}}\cap [J,J+1)=\{\lambda_{2^{J}-1}, \lambda_{2^{J}-1+1},...,
\lambda_{2^{J}-1+2^{J}-1}\} $ for any $J\in  {\ensuremath {\mathbb N}}$.

We say that ${{\omega}}$ is {\it $J$-$L'$-good}
if
\begin{equation}\label{*wlaw2}
\Big |\frac{\sum_{j=0}^{2^{J-L'}-1} X_{2^{J}-1+l\cdot 2^{J-L'}+j}( {\omega})}{2^{J-L'}}-p\Big |\leq p(1-2^{-1.9})
\end{equation}
$$\text{  holds for every $l=0,...,2^{L'  }-1$}.$$
If ${{\omega}}$ is $J$-$L'$-good for every $J\geq J_{0}$ then we say that
 $ {\omega}$ is {\it $J_{0}$-$L'$-$ {\infty}$-good}.

By \eqref{*wlaw} one can see that
\begin{equation}\label{*wlaw3}
\P\{ {\omega}:  {\omega} \text{  is $J$-$L'$-good} \}\geq 1- 2^{L'}\cdot \frac{C_{p}}{2^{J-L'}}
\end{equation}
and hence
\begin{equation}\label{*wlaw4}
\P\{ {\omega}:  {\omega} \text{  is $J_{0  }$-$L'$-$ {\infty}$-good} \}\geq 1- \sum_{J\geq J_{0}} 2^{L'}\cdot \frac{C_{p}}{2^{J-L'}}.
\end{equation}

It is easy to see that if  $L'$ is sufficiently large, say
$L'=L+100$ and $ {\Lambda}= {\Lambda}_{{\omega}}$ for a $J_{0}$-$L'$-$ {\infty}$-good $ {\omega}$,
then \eqref{p dense} holds with $N=J_{0}$.

Using \eqref{*wlaw4} it is also clear that for $\P$
a.e. $ {\omega}$ there is a $J_{0}$ such that $ {\omega}$ is $J_{0}$-$L'$-$ {\infty}$-good.
This completes the proof of the lemma.
\end{proof}

\bigskip

\begin{proof}[Proof of Theorem \ref{prob p theorem}]

Let $0 < p <1$ and assume that $\Lambda$ has been chosen with probability $p$ from $\widetilde{\Lambda}$.

Pursuing a contradiction, we assume that $\Lambda$ is type 2.

By Theorem \ref{*thwit} we can choose a measurable set $S \subset  {\mathbb R}$ such that $f=\mathbf{1}_S$ witnesses that $\Lambda $ is type 2.  Thus $\mu(D(f,\Lambda))>0$ and $\mu(C(f,\Lambda))> 0$ and therefore we can choose $R\in  {\mathbb {N}}$ and an interval $I$ of length $R-1$
 such that $\mu(D(f,\Lambda)\cap I)>0$ and $\mu(C(f,\Lambda)\cap I)>0$.  
 Then using the Lebesgue Density Theorem we choose intervals $I_{\pmb{D}}$ and $I_{\pmb{C}}$ subsets of $I$ of length $2^{-L}$ where $L \in  {\mathbb {N}}$ such that
 \begin{equation}\label{IC big}
\mu(I_{\pmb{C}} \cap C(f,\Lambda)) > \Big (1-\frac{p}{2^{R+7}}\Big )\cdot 2^{-L}
 \end{equation}
and
 \begin{equation}\label{ID big}
\mu(I_{\pmb{D}} \cap D(f,\Lambda))> 0.
 \end{equation}
We assume without loss of generality that $I_{\pmb{C}}=[0,\frac1{2^L})$ and $I_{\pmb{D}}=[\frac{-N}{2^L},\frac{-(N-1)}{2^L})$ for some $N \in  {\ensuremath {\mathbb Z}}$.
Since the cases $N\leq 0$ are easier than the ones when $N>0$
we provide details only for the case $N\in  {\ensuremath {\mathbb N}}$.

  Note that we have
  \begin{equation}\label{*NRL}
  N \le R \cdot 2^L.
  \end{equation}

For each $n\in  {\mathbb {N}}$ we define $C^*_n=\{x \in C(f,\Lambda)\cap I_{\pmb{C}} : \, (x+\Lambda) \cap [n,\infty) \cap S =\emptyset\}$.  Since $f$ is a characteristic function, it follows that $\cup_{n=1}^\infty C^*_n=C(f,\Lambda)\cap I_{\pmb{C}}$ and therefore we can choose $C \subset C(f,\Lambda)\cap I_{\pmb{C}}$ and $M \in  {\mathbb {N}}$ such that
 \begin{equation}\label{C big}
 \mu(C) \ge \Big (1-\frac{p}{2^{R+6}}\Big )\cdot 2^{-L}
  \end{equation}
and
 \begin{equation}\label{C intersect empty}
(C +\Lambda)\cap [M\cdot 2^{-L},\infty)\cap S=\emptyset.
 \end{equation}
For each $n \ge L$ define $n^*= \lfloor \frac{n}{2^L}  \rfloor$ and let
$$ C_n =\Big \{x \in I_{\pmb{C}}: \, \#((x+2^{-k} {\mathbb {Z}})\cap C)>\Big (1-\frac{p}{2^{R+5}}\Big )2^{k-L} \text{    for all    }k \ge (n+N)^* \Big \}$$
and $E_n=C_n-\frac{N}{2^L} {\subset}  I_{\pmb{D}}$.
Note that for all $n \ge L$ we have $C_n \subset C_{n+1} \subset I_{\pmb{C}}$ and by a rescaled version of Lemma \ref{*lemmart} we know that
\begin{equation}\label{*IDEn}
\text{  $\mu(C_n) \to 2^{-L  }= {\mu}(I_{\pmb{C}})$ and hence $ {\mu}( I_{\pmb{D}} {\setminus} E_{n})\to 0.$}
\end{equation}
Note also that $C_n$ is $\frac1{2^{(n+N)^*}}$ periodic on $I_{\pmb{C}}$ and $E_n$ is $\frac{1}{2^{(n+N)^*}}$ periodic on $I_{\pmb{D}}$ for all $n \ge L$.

For each $n \in  {\mathbb {N}}$ define
$$ S_n=\Big \{y \in \Big [\frac{n} {2^{L}},\frac{n+1} {2^{L}}\Big ): \, (y-\Lambda)\cap C_{n} =\emptyset\Big \},$$
and
$$S_n^\prime=\Big \{y \in \Big [\frac{n} {2^{L}},\frac{n+1} {2^{L}}\Big ): \, (y-\Lambda)\cap C =\emptyset \Big \}.$$

Using Lemma \ref{prob p lemma}, we may assume that we can choose $P \in  {\mathbb {N}}$ such that
 \begin{equation}\label{p dense lambda}
\#\Big (\Lambda\cap \Big [x,x+\frac{1}{2^L}\Big )\Big ) > {p}\cdot{2^{J-L-2}} \text{    for all    }x\ge P,
 \end{equation}
where $J= \lfloor x \rfloor$.

Next we show that
\begin{equation}\label{*CA45}
\text{  if $n >B:= \max \{M,(P+L )\cdot 2^L+1 \}$, then $S \cap \Big
[\frac{n}{2^L},\frac{n+1}{2^L}\Big )\subset S_n^\prime \subset S_n$.
}
\end{equation}
Assume that $n > B$.  Since $n > M$,
by \eqref{C intersect empty} we have $(C + \Lambda) \cap S \cap [\frac{n}{2^L},\frac{n+1}{2^L}) =\emptyset$ and therefore
$S \cap [\frac{n}{2^L},\frac{n+1}{2^L}) \subset S^\prime_{n}$.

Now suppose that $y \in  [\frac{n}{2^L},\frac{n+1}{2^L}) {\setminus} S_n$.  Then we can choose $x \in C_n$ and $\lambda \in \Lambda$ such that $x+\lambda=y$.
Since $\lambda=y-x<\frac{n+1}{2^{L}}\leq n^{*}+1$
by \eqref{*tL} and $ {\Lambda} {\subset}  {{\widetilde {\Lambda}}}$
 we have $\lambda\in 2^{-n^{*}} {\ensuremath {\mathbb Z}}$ which implies
\begin{equation}\label{*xylll}
\text{  $y+2^{-n^{*  }} {\ensuremath {\mathbb Z}}=x+2^{-n^{*}} {\ensuremath {\mathbb Z}}$ and $(y- {\Lambda})\cap \Big [0,\frac{1}{2^L}\Big ) {\subset}
(y+2^{-n^{*}} {\ensuremath {\mathbb Z}})\cap \Big [0,\frac{1}{2^L}\Big )$.}
\end{equation}

From the definition of $C_n$ we have that
$$\#((x+2^{-(n+N)^*} {\mathbb {Z}}) \cap C) > \Big (1-\frac{p}{2^{R+5}}\Big )\cdot 2^{(n+N)^*-L},\text{  that is  }$$
$$\frac{p}{2^{R+5}}\cdot 2^{(n+N)^*-L}>\#\Big (\Big ((x+2^{-(n+N)^*} {\mathbb {Z}}) \cap \Big  [0,\frac{1}{2^L}\Big )\Big ) {\setminus} C\Big )  $$
$$\geq \#\Big (\Big ((x+2^{-n^*} {\mathbb {Z}}) \cap \Big  [0,\frac{1}{2^L}\Big )\Big ) {\setminus} C\Big ).$$
It follows that
$$\#((x+2^{-n^*} {\mathbb {Z}})\cap C) > 2^{{n^*}-L}-\frac{p}{2^{R+5}}(2^{(n+N)^*-L})=2^{n^*-L}
\Big (1-\frac{p}{2^{R+5}}2^{(n+N)^*-n^*}\Big ).$$
Using  $(n+N)^*-n^*\le N^*+1$ and by \eqref{*NRL}, $R \ge N^*$, we conclude that
 \begin{equation}\label{1-pover2}
\#((y+2^{-n^*} {\mathbb {Z}})\cap C)=\#((x+2^{-n^*} {\mathbb {Z}})\cap C) > \Big (1-\frac{p}8 \Big )\cdot 2^{n^*-L}.
 \end{equation}
Now using \eqref{p dense lambda} with $y-\frac1{2^L}$ in place of $x$ we find that since $y-\frac1{2^L}\geq \frac{n-1}{2^{L}}\geq P$
$$\#\Big ((y-\Lambda) \cap \Big [0,\frac1{2^L}\Big )\Big )=
\#\Big (\Lambda \cap \Big [y-\frac1{2^L},y\Big )\Big )>p\cdot 2^{n^*-1-L-2}=\frac{p}8\cdot 2^{n^*-L}.$$
Using \eqref{*xylll}, \eqref{1-pover2},
the last inequality and the pigeon-hole
principle, we conclude that there
must exist $x^\prime \in C$ and $\lambda^\prime \in \Lambda$ such that $x^\prime+\lambda^\prime=y$ and therefore $y \notin S_n^\prime$.   It follows that $S_n^\prime \subset S_n$ and we are done with the proof of \eqref{*CA45}.

Next we continue with some definitions.
For each $n \in  {\mathbb {N}}$ we define $D_n=(S_n-\Lambda)\cap I_{\pmb{D}}$ and let $D_n^\prime=D_n \cap E_n$ and $D_n^{\prime\prime}=D_n \backslash D_n^\prime$.  Note that if $x \in I_{\pmb{D}} \backslash D_n$ and $n> B$, then
$(x+\Lambda)\cap S_n^\prime =\emptyset$ so $f(x+\lambda)=0$ for all $\lambda \in [\frac{n+N}{2^L}-x,\frac{n+N+1}{2^L}-x)$.
From these considerations it follows that
$$ D(f,\Lambda) \cap I_{\pmb{D}} \subset \cap_{k=1}^\infty \cup_{n=k}^\infty D_n =
(\cap_{k=1}^\infty \cup_{n=k}^\infty D_n^\prime)\cup (\cap_{k=1}^\infty \cup_{n=k}^\infty D_n^{\prime\prime}).
$$
Furthermore, for all $n \in  {\mathbb {N}}$ we have $D_n^{\prime\prime}\subset I_{\pmb{D}}\backslash E_n$ where $I_{\pmb{D}}\backslash E_{n+1} \subset I_{\pmb{D}}\backslash E_n$ and
by \eqref{*IDEn},
 $\mu(I_{\pmb{D}}\backslash E_n)\to 0$ and therefore
 \begin{equation}\label{Tnprime}
\mu(\cap_{k=1}^\infty \cup_{n=k}^\infty D_n^{\prime\prime})=0.
 \end{equation}

Thus, if we can prove that $\mu(\cap_{k=1}^\infty \cup_{n=k}^\infty D_n^\prime)=0$ we can conclude that
$\mu(D(f,\Lambda) \cap I_{\pmb{D}})=0$, which contradicts (\ref{ID big}) and finishes the proof of the theorem.
Actually we prove that  $D_{n}'= {\emptyset}$ for large $n$.

Suppose that $n > B $. Then $n^*\geq L$ and $\lambda''=\frac{n}{2^L}\in 2^{-n^*} {\ensuremath {\mathbb Z}}.$
We show that
\begin{equation}\label{*CE45}
\Big (C_{n}+\frac{n}{2^L}\Big )\cap S_{n}=(C_{n}+\lambda'')\cap S_{n}= {\emptyset}.
\end{equation}
Indeed, suppose that $y=x+\lambda'' $ with $x\in C_{n}$ then we show that one can find
$x'\in C_{n}$ and $\lambda'\in {\Lambda}$ such that $y=x'+\lambda'$
and hence $y\not \in S_{n}$. This follows easily, since
$C_{n}$ is $\frac1{2^{n^*}}$ periodic on $I_{\pmb{C}}$ and
one can apply (\ref{p dense lambda}) for $x''=x+\frac{n-1}{2^L}$ and
observe that there are points of $ {\Lambda}$ in $ (x+\frac{n-1}{2^L},x+\frac{n}{2^L})$.
Select such a point $\lambda'.$ Then $\lambda''-\lambda'\in 2^{-n^*} {\ensuremath {\mathbb Z}}$
and hence if we let $x'=x+(\lambda''-\lambda')$ then $y=x'+\lambda'$
and $x'\in [0,\frac{1}{2^L})$. Since $C_{n}$ is $2^{-n^*}$ periodic
in $[0,\frac{1}{2^L})$ we obtained that $x'\in C_{n}$, proving
\eqref{*CE45}.

 Now recall that  $C_n$ is $\frac1{2^{(n+N)^*}}$ periodic on $I_{\pmb{C}}$. This implies that $C_n+\lambda''=C_{n}+\frac{n}{2^L}$ is $\frac1{2^{(n+N)^*}}$ periodic on $[\frac{n} {2^{L}},\frac{n+1} {2^{L}})$.
 By \eqref{*CE45} $$S_{n} {\subset}  {{\widetilde {C}}}_{n}:=\Big [\frac{n} {2^{L}},\frac{n+1} {2^{L}}\Big ) {\setminus} \Big (C_{n}+\frac{n}{2^L}\Big ).$$
Obviously $ {{\widetilde {C}}}_{n}$ is also $\frac1{2^{(n+N)^*}}$ periodic on $[\frac{n} {2^{L}},\frac{n+1} {2^{L}})$.  Since we also know that $E_n=C_n-\frac{N}{2^L}$ is $\frac1{2^{(n+N)^*}}$ periodic on 
$[\frac{-N}{2^L},\frac{-(N-1)}{2^L})$, it follows that 
\begin{equation}\label{*tCn}
( {{\widetilde {C}}}_{n}-2^{-(n+N)^*} {\ensuremath {\mathbb N}})\cap E_{n}= {\emptyset}.
\end{equation}
Now observe that if $y\in \widetilde{C}_n$ and $y-\lambda \in E_n$, then we have $\lambda \in [\frac{N+n-1}{2^L},\frac{N+n+1}{2^L})$.  Moreover, we also have that $\lambda \in [\frac{N+n-1}{2^L},\frac{N+n+1}{2^L}) \cap \Lambda$ implies that $\lambda \in 2^{-(n+N)^*} {\ensuremath {\mathbb N}}$.  
Since $S_n \subset \widetilde{C}_n$, it follows from (\ref{*tCn}) that $(S_n-\Lambda)\cap E_{n}= {\emptyset}$, which implies that
 $D_{n}'= {\emptyset}$ for $n> B.$
 This concludes the proof of Theorem \ref{prob p theorem}.
  \end{proof}

\begin{theorem}\label{*randt2} Suppose that $(m_k)$ and $(n_k)$ are strictly increasing sequences of
positive integers. For each $k\in\mathbb{N}$, define $\Lambda_k=2^{-m_k} \mathbb{N} \cap [n_k, n_{k+1})$ and let $ {{\widetilde {\Lambda}}}=\bigcup_{k=1}^{\infty}\Lambda_k$. Moreover, fix $0<p<1$ and suppose that $\Lambda$ is chosen  with probability $p$ from $ {{\widetilde {\Lambda}}}$.
Set $q=1-p$.
 For fixed $(m_k)$, if $(n_k)$ tends to infinity sufficiently fast
 then  almost surely $ {\Lambda}$ is type $2$. Notably, if the series $\sum_{k=1}^{\infty}1-\left(1-{q}^{2^{m_k}}\right)^{n_{k+1}-n_k}$ diverges then  almost surely $ {\Lambda}$ is type $2$. \end{theorem}

 \begin{remark}
 If $m_{k}=k$ then by Example \ref{*exdyada},
$\Lambda$ is type $1$ for any $n_{k}$ and hence
it may happen that a type $1$ set is turned into a type $2$
set by random deletion of its elements.
 \end{remark}
 
\begin{proof}
Let $A_k$ denote the event in which there exists
$a\in {\ensuremath {\mathbb N}}$ such that
$[a,a+1) {\subset} [n_k,n_{k+1})$ and $[a,a+1)\cap {\Lambda}=  {\emptyset}$.
 We can quickly deduce that the probability of the complement is
$$
\mathbb{P}(A_k^c)=\left(1-{q}^{2^{m_k}}\right)^{n_{k+1}-n_k}.
$$
Consequently,
$$
\mathbb{P}(A_k)=1-\left(1-{q}^{2^{m_k}}\right)^{n_{k+1}-n_k}.
$$
By assumption, the series of these probabilites diverges. Consider now the sequence of events $(A_k)_{k=1}^{\infty}$. They are clearly independent, hence by the second Borel--Cantelli lemma the aforementioned divergence implies that almost surely infinitely many of the events $A_k$ occurs.
 However, this immediately yields that
almost surely the set
$ {\Lambda}$ is asymptotically
lacunary  and hence by Theorem \ref{*thlacu}
 it is type $2$.
\end{proof}

\section{$\pmb{c}$-type 1 and 2 sets}
\label{*secpmbc}

 The following theorem is a  nice consequence of Theorem \ref{*thwit}.

 \begin{theorem}\label{*thpmbc1}
 If a set $\Lambda$ is $\pmb{1}$-type $2$, then it is $\pmb{c}$-type $2$ with respect to any positive sequence $\pmb{c}=(c_n)_{n=1}^{\infty}$. \end{theorem}

\begin{proof} By Theorem \ref{*thwit}, choose an open set such that its characteristic function $f$ witnesses that $\Lambda$ is
$\pmb{1}$-type $2$.
 Then both $D(f,\Lambda)$ and $C(f,\Lambda)$ have positive measure. Choose a bounded $D\subset D(f,\Lambda)$ and a bounded $C\subset C(f,\Lambda)$ of positive measure. Then the set $\{f\neq{0}\}$ equals a countable union of intervals $I_1,I_2,...$. We will construct $g$ verifying the statement such that for any $x\in{I_k}$, $k=1,2,...$ we have $g(x)=\alpha_k f(x)$ for some $\alpha_k>0$. We define $\alpha_k$ as follows: since $D$ is bounded, for any $k=1,2,...$ there are finitely many $\lambda_{k_1},...,\lambda_{k_m}$ such that $x+\lambda_{k_i}\in{I_k}$ for some $x\in{D}$ and $i=1,...,m$. As $c_n>0$ for each $n$, we have that the finite set $\{c_{k_1},...,c_{k_m}\}$ is bounded away from 0. Thus $\alpha_k$ can be chosen sufficiently large to guarantee $\alpha_k c_{k_i}\geq{1}$ for $i=1,...,m$. By this choice for any $x\in{D}$ we have that
$$
\sum_{\lambda_{j}\text{  :  }x+\lambda_{j}\in I_k}c_j g(x+\lambda_{j})\geq \sum_{\lambda_{j}\text{  :  }x+\lambda_{j}\in I_k}f(x+\lambda_{j}).
$$
However, if we add these latter sums for all the intervals $I_k$, we find that our sum diverges. As a consequence, $\sum_{n=1}^{\infty}c_n g(x+\lambda_n)$ diverges for any $x\in{D}$, which guarantees the positive measure of the divergence set.

Concerning the convergence set, we have an easy task: for any $x\in{C}$ we have that $x+\lambda\in \{f\neq{0}\}$ only for finitely many $\lambda$s since otherwise $\sum_{n=1}^{\infty} f(x+\lambda_n)$ would diverge as $\{f\neq{0}\}=\{f={1}\}$. Thus we also have $x+\lambda\in \{g\neq{0}\}$ only for finitely many $\lambda$s. This guarantees that $\sum_{n=1}^{\infty}c_n g(x+\lambda_n)$ converges for any $x\in{C}$, which guarantees the positive measure of the convergence set. \end{proof}

The previous theorem displays that the sequence $c_n\equiv 1$ is minimal in some sense: the family of type $2$ sets is as small as possible. It is natural to ask whether all $\chi$-sequences have this property.

Theorem \ref{*univt2seq} shows that not all sequences have this property by showing the other extreme: sequences for which every $ {\Lambda}$ is $\pmb{c}$-type  $2$.

\begin{theorem}\label{*univt2seq}
Suppose that $\pmb{c}=(c_n)$ is a sequence of positive numbers satisfying the following condition:
 \begin{equation}\label{cn fast decr}
\sum_{j=n+1}^\infty c_j < 2^{-n}c_n \text{    for every    } n \in  {\mathbb {N}}.
\end{equation}
Then every discrete set $\Lambda$ is $\pmb{c}$-type  $2$.
\end{theorem}

\begin{proof} Let $\Lambda=\{\lambda_1,\lambda_2,\dots \}$ with $\lambda_1<\lambda_2< \dots $ and $\lambda_n \to \infty$.
Choose $y_n \nearrow \infty$ such that  $y_{n+1}-y_n > 1$ and $\Lambda \cap [y_n,y_n+\frac12] \neq \emptyset $ for all $n \in  {\mathbb {N}}$. For each $n \in  {\mathbb {N}}$ let
$$T_n =\Big \{j : \ \lambda_j\in \Big [y_n,y_n+\frac12 \Big ]\Big \},$$
and define
\begin{equation}\label{*dnf}
d_n = \frac1{\sum_{j\in T_n} c_j} \text{    and    } f=\sum_{n=1}^\infty d_n \mathbf{1}_{[y_n,y_n+1]}.
\end{equation}

\begin{claim}\label{*cll1}
$[0,\frac12]\subset D_{\pmb{c}}(f,\Lambda)$.
\end{claim}

\begin{proof} Let $x \in [0,\frac12]$.  Then for every $j \in T_n$ we have $x+\lambda_j\in [y_n,y_n+1]$. Hence
$f(x+\lambda_j)=d_n$.  Thus we obtain
$$\sum_{j=1}^\infty c_j f(x+\lambda_j) \ge \sum_{n=1}^\infty \sum_{j \in T_n} c_j f(x+\lambda_j)=\sum_{n=1}^\infty \frac1{d_n}d_n=\infty,$$ finishing the proof of Claim
\ref{*cll1}.
\end{proof}
\bigskip

\begin{claim}\label{*cll2}
$(-\infty,-\frac12) \subset C_{\pmb{c}}(f,\Lambda).$
\end{claim}

\begin{proof} Let $x \in (-\infty,-\frac12)$.  For each $n \in  {\mathbb {N}}$ define
$$R_n(x)=\{j: \ x+\lambda_j \in [y_n,y_n+1]\}.$$
  Note that if $j \in R_n(x)$, then $\lambda_j > y_n+\frac12$ and it follows that $j>i$ for all
$i \in T_n$.  Therefore, using (\ref{cn fast decr}) we obtain
$$\sum_{j\in R_n(x)} c_j < 2^{-n}\sum_{j\in T_n} c_j.$$  Therefore using \eqref{*dnf} we deduce
$$\sum_{j=1}^\infty c_j f(x+\lambda_j)=\sum_{n=1}^\infty \sum_{j \in R_n(x)} c_j f(x+\lambda_j) = \sum_{n=1}^\infty \Big (d_n \sum_{j\in R_n(x)} c_j \Big )$$
$$< \sum_{n=1}^\infty d_n 2^{-n}\sum_{j\in T_n} c_j=\sum_{n=1}^\infty 2^{-n}=1.$$
\end{proof}
Clearly the proof of Claim \ref{*cll2} also concludes the proof of Theorem \ref{*univt2seq}.
\end{proof}

\section{Acknowledgements}

Z. Buczolich thanks the R\'enyi Institute where he was
a visiting researcher for the academic year 2017-18.

B. Hanson would like to thank the Fulbright Commission and the R\'enyi Institute for their generous support during the Spring of 2018, while he was visiting Budapest as a Fulbright scholar.


\end{document}